\documentclass[11pt]{article}

\usepackage{amsmath,amssymb,amsthm,pifont}

\title{Effectively constructible fixed points in Sacchetti's modal logics of provability}
\author{Taishi Kurahashi and Yuya Okawa}
\date{}

\theoremstyle{plain}
\newtheorem{theorem}{Theorem}[section]
\newtheorem{lem}[theorem]{Lemma}
\newtheorem{prop}[theorem]{Proposition}

\newtheorem{cor}[theorem]{Corollary} 
\newtheorem{fact}[theorem]{Fact} 
\newtheorem{prob}[theorem]{Problem} 

\theoremstyle{definition}
\newtheorem{defn}[theorem]{Definition}
\newtheorem{ex}[theorem]{Example}

\theoremstyle{remark}

\newcommand{\PA}{{\sf PA}}
\newcommand{\K}{{\bf K}}
\newcommand{\GL}{{\bf GL}}
\newcommand{\wGL}{{\bf wGL}}
\newcommand{\At}{{\sf Var}}
\newcommand{\dep}{{\sf dep}}
\newcommand{\db}[1]{\ensuremath{[\![#1]\!]} }

\begin{document}
\maketitle

\abstract{
We give a purely syntactical proof of the fixed point theorem for Sacchetti's modal logics $\K + \Box(\Box^n p \to p) \to \Box p$ ($n \geq 2$) of provability. 
From our proof, an effective procedure for constructing fixed points in these logics is obtained. 
We also show the existence of simple fixed points for particular modal formulas.
}

\section{Introduction}\label{intro}

Solovay's arithmetical completeness theorem \cite{Sol76} states that the propositional modal logic $\GL$ is the provability logic of the standard G\"odel provability predicate ${\rm Pr}_{\PA}(x)$ of Peano Arithmetic $\PA$, that is, for any modal formula $A$, $A$ is provable in $\GL$ if and only if $A$ is provable in $\PA$ under any arithmetical interpretation where $\Box$ is interpreted as ${\rm Pr}_{\PA}(x)$. 
From Solovay's theorem, some aspects of metamathematics of $\PA$ may be reflected in $\GL$. 
In fact, metamathematical facts about self-reference are already provable in $\GL$, that is, the fixed point theorem holds for $\GL$. 

A modal formula $A$ is said to be modalized in $p$ if all occurrences of $p$ in $A$ are under the scope of $\Box$. 
We say that a modal formula $F$ is a fixed point of a modal formula $A(p)$ in $\GL$ if $p$ does not appear in $F$, all propositional variables appearing in $F$ are already in $A$ and $\GL \vdash F \leftrightarrow A(F)$. 
The fixed point theorem for $\GL$ states that for any modal formula $A(p)$ which is modalized in $p$, there exists a fixed point $F$ of $A(p)$ in $\GL$ (see also \cite{Lin96,Lin06,Rei90,SV82}).  
The fixed point theorem for $\GL$ was independently proved by de Jongh and Sambin \cite{Sam76} as one of early achievements of the investigations of provability logic. 
Sambin's proof is purely syntactical, and gives an effective procedure for constructing fixed points in $\GL$. 

The fixed point theorem for weaker modal logics were investigated by Sacchetti \cite{Sac01}. 
Sacchetti introduced the modal logics $\wGL_n = \K + \Box(\Box^n p \to p) \to \Box p$ for $n \geq 2$ which are weaker than $\GL$ (here ``w'' stands for ``weak''), and proved the fixed point theorem for these modal logics. 
These modal logics are actually provability logics for some nonstandard provability predicates, that is, Kurahashi \cite{Kur18} proved that for each $n \geq 2$, there exists a $\Sigma_2$ provability predicate such that $\wGL_n$ is sound and complete with respect to the arithmetical interpretation based on the provability predicate. 
Therefore we may say that metamathematical aspects of arithmetic are also reflected in these logics. 
In fact, for example, since the uniqueness of fixed points is proved in $\wGL_n$ (see Sacchetti \cite[Proposition 3.6]{Sac01}), the corresponding arithmetical statement written using such nonstandard provability predicate is also provable in arithmetic.

Sacchetti's proof of the fixed point theorem is based on Smory\'nski's semantical argument \cite{Smo78}, and gives no effective procedure for constructing fixed points in $\wGL_n$. 
Then Sacchetti asked the following question (see the last paragraph of Sacchetti \cite{Sac01}):
 
\begin{quotation}
{\small
Is there a constructive proof of the fixed point theorem for $\K + \Box(\Box^{n-1} p \to p) \to \Box p$? 
}
\end{quotation}

In this paper, we solve Sacchetti's question affirmatively, that is, we provide a purely syntactical proof of the fixed point theorem for $\wGL_n$, and effectively constructible fixed points in $\wGL_n$ are obtained from our proof. 
As a result, in arithmetic, we can effectively construct concrete fixed points of particular formulas of arithmetic written using the nonstandard provability predicates introduced in \cite{Kur18}.

\section{Preliminaries}

In this paper, we assume that the logical symbols in the language of propositional modal logic are $\bot$, $\to$ and $\Box$, and other symbols such as $\top$, $\land$ and $\Diamond$ are defined from these symbols in a usual way. 
For each modal formula $A$, the set of propositional variables contained in $A$ is denoted by $\At(A)$. 
The axioms of the modal logic $\K$ are Boolean tautologies in the language of propositional modal logic and the modal formula $\Box (p \to q) \to (\Box p \to \Box q)$. 
The inference rules of $\K$ are modus ponens, necessitation and substitution. 
For a modal formula $A$, let $\K + A$ denote the logic axiomatized by adding a new axiom $A$ to $\K$. For each natural number $k$, we define the expression $\Box^k A$ inductively as follows: 
$\Box^0 A \equiv A$ and $\Box^{k+1} A \equiv \Box \Box^k A$. 
Then the modal logics $\GL$ and $\wGL_n$ for $n \geq 2$ are defined as follows: 
\begin{enumerate}
	\item $\GL : = \K + \Box(\Box p \to p) \to \Box p$. 
	\item $\wGL_n : = \K + \Box(\Box^n p \to p) \to \Box p$. 
\end{enumerate}

The modal logic $\GL$ is known as the modal logic of provability (see \cite{Boo93,Smo85}). 
We say that a modal formula $A(p)$ is {\it modalized in $p$} if all occurrences of $p$ in $A$ are under the scope of $\Box$. 
The fixed point theorem for $\GL$ was independently proved by de Jongh and Sambin \cite{Sam76}. 

\begin{theorem}[The fixed point theorem for $\GL$]
For any modal formula $A(p)$ which is modalized in $p$, there exists a modal formula $F$ such that $\GL \vdash F \leftrightarrow A(F)$ and $\At(F) \subseteq \At(A) \setminus \{p\}$. 
\end{theorem}

Such a modal formula $F$ is said to be a {\it fixed point of $A(p)$ in $\GL$}. 
Sambin's proof of the fixed point theorem is purely syntactical, and then we can extract an algorithm for constructing fixed points in $\GL$ from his proof. 
Such an algorithm is said to be Sambin's algorithm. 

\begin{theorem}[Sambin \cite{Sam76}]
For any modal formula $A(p)$ which is modalized in $p$, a fixed point of $A(p)$ in $\GL$ is effectively constructible. 
\end{theorem}

The fixed point theorem is not specific to $\GL$. 
Sacchetti \cite{Sac01} introduced the modal logics $\wGL_n$ and proved the fixed point theorem for these logics. 

\begin{theorem}[Sacchetti's fixed point theorem \cite{Sac01}]
Let $n \geq 2$. 
Then for any modal formula $A(p)$ which is modalized in $p$, there exists a fixed point of $A(p)$ in $\wGL_n$.
\end{theorem}

Sacchetti's proof is based on Smory\'nski's proof \cite{Smo78} of the fixed point theorem for $\GL$, and does not give an effective construction of fixed points in $\wGL_n$. 
Sacchetti proposed the following problem. 

\begin{prob}[Sacchetti \cite{Sac01}]
Is there a constructive proof of the fixed point theorem of $\wGL_n$ for $n \geq 2$?
\end{prob}

The main purpose of this paper is to give an affirmative answer to this problem. 

We denote by $\db{k} A$ and $\db{k}^+ A$ the formulas $\Box A \land \Box^2 A \land \cdots \land \Box^k A$ and $A \land \db{k} A$, respectively.
For any modal formula $A(p)$, we define the modal formula $(A)^k(p)$ for each $k \in \omega$ recursively as follows: 
\begin{enumerate}
	\item $A^0(p) \equiv p$; 
	\item $A^{k+1}(p) \equiv A(A^k(p))$.
\end{enumerate}

For each occurrence of a propositional variable $p$ in a modal formula $A$, the number of subformulas of the form $\Box B$ of $A$ containing the occurrence is said to be the (modal) depth of the occurrence in $A$. 
Moreover, we define the set $\dep(A, p)$ of (modal) depths of all occurrences of a propositional variable $p$ in a modal formula $A$. 

\begin{defn}
For any modal formula $A$ and any propositional variable $p$, we define the set $\dep(A, p) \subseteq \omega$ recursively as follows: 
\begin{enumerate}
	\item If $A$ is $p$, then $\dep(A, p) = \{0\}$; 
	\item If $A$ is a propositional variable $q \not \equiv p$ or $\bot$, then $\dep(A, p) = \emptyset$;
	\item If $A$ is of the form $B \to C$, then $\dep(A, p) = \dep(B, p) \cup \dep(C, p)$;
	\item If $A$ is of the form $\Box B$, then $\dep(A, p) = \{x + 1 : x \in \dep(B, p)\}$. 
\end{enumerate}
\end{defn}

Moreover, in considering fixed points in $\wGL_n$, the set of all depths of occurrences of $p$ in $A$ modulo $n$ plays an important role. 
For each $x \in \omega$, let $[x]_n : = \{y \in \omega : y$ is congruent to $x$ modulo $n\}$. 

\begin{defn}
For any modal formula $A$ and any propositional variable $p$, define $\dep_n(A, p)$ to be the set $\{[x]_n : x \in \dep(A, p)\}$. 
\end{defn}


\begin{ex}
The depths of occurrences of $p$ from left to right in the modal formula $A \equiv p \land \Box(p \to \Box^2 p)$ are $0$, $1$ and $3$, respectively. 
Also $\dep(A, p) = \{0, 1, 3\}$ and $\dep_3(A, p) = \{[0]_3, [1]_3\}$. 
\end{ex}

We prove some lemmas concerning the sets $\dep(A, p)$ and $\dep_n(A, p)$. 

\begin{lem}\label{lem1}
For any modal formulas $A(p)$ and $B$, 
\[
	\dep(A(B), p) = \{x + y : x \in \dep(A, p)\ \text{and}\ y \in \dep(B, p)\}.
\]
\end{lem}
\begin{proof}
We prove by induction on the construction of $A$. 

If $A$ is $p$, then $A(B) \equiv B$ and $\dep(A, p) = \{0\}$. 
The statement follows obviously. 

If $A$ is a propositional variable $q \not \equiv p$ or $\bot$, then $A(B) \equiv A$ and $\dep(A, p) = \emptyset$. 
Then the statement is trivial. 

If $A$ is of the form $C_0 \to C_1$, and suppose that the statement holds for $C_0$ and $C_1$. 
\begin{align*}
	\dep((C_0 \to C_1)(B), p) & = \dep(C_0(B), p) \cup \dep(C_1(B), p) \\
	& = \bigcup_{i = 0, 1}\{x + y : x \in \dep(C_i, p), y \in \dep(B, p)\}\\
	& = \{x + y : x \in \dep(C_0, p) \cup \dep(C_1, p), y \in \dep(B, p)\}\\
	& = \{x + y : x \in \dep(A, p), y \in \dep(B, p)\}.
\end{align*}

If $A$ is of the form $\Box C$ and suppose that the statement holds for $C$. 
\begin{align*}
	\dep(\Box C(B), p) & = \{x + 1 : x \in \dep(C(B), p)\} \\
	& = \{x + y + 1 : x \in \dep(C, p), y \in \dep(B, p)\}\\
	& = \{x + y : x \in \dep(A, p), y \in \dep(B, p)\}.
\end{align*}
\end{proof}

\begin{lem}\label{lem2}
For any modal formulas $A(p, q)$ and $B$, 
\[
	\dep(A(p, B), p) = \{x + y : x \in \dep(A, q)\ \text{and}\ y \in \dep(B, p)\} \cup \dep(A, p).
\]
\end{lem}
\begin{proof}
We prove by induction on the construction of $A$. 

If $A$ is $p$, then $A(p, B) \equiv A$, $\dep(A, p) = \{0\}$ and $\dep(A, q) = \emptyset$. 
Hence the statement holds. 

If $A$ is $q$, then $A(p, B) \equiv B$, $\dep(A, p) = \emptyset$, $\dep(A, q) = \{0\}$. 
Then the statement is true. 

If $A$ is $r \not \equiv p, q$ or $\bot$, then $A(p, B) = A$, $\dep(A, p) = \dep(A, q) = \emptyset$. 
This case is also trivial. 

If $A$ is of the form $C_0 \to C_1$, and suppose that the statement holds for $C_0$ and $C_1$. 
\begin{align*}
	& \dep((C_0 \to C_1)(p, B), p)\\
	= & \dep(C_0(p, B), p) \cup \dep(C_1(p, B), p) \\
	= & \bigcup_{i = 0, 1} \left(\{x + y : x \in \dep(C_i, q), y \in \dep(B, p)\} \cup \dep(C_i, p) \right)\\
	= & \{x + y : x \in \dep(A, q), y \in \dep(B, p)\} \cup \dep(A, p).
\end{align*}

If $A$ is of the form $\Box C$ and suppose that the statement holds for $C$. 
\begin{align*}
	& \dep(\Box C(p, B), p)\\
	= & \{x + 1 : x \in \dep(C(p, B), p)\} \\
	= & \{x + y + 1 : x \in \dep(C, q), y \in \dep(B, p)\} \cup \{x + 1 : x \in \dep(C, p)\}\\
	= & \{x + y : x \in \dep(A, q), y \in \dep(B, p)\} \cup \dep(A, p).
\end{align*}
\end{proof}

\begin{lem}\label{lem3}
For any modal formula $A(p)$, 
\[
	\dep(\Box A, p) = \dep(A(\Box p), p).
\]
\end{lem}
\begin{proof}
By Lemma \ref{lem1}, 
\begin{align*}
	\dep(A(\Box p), p) & = \{x + y : x \in \dep(A, p)\ \text{and}\ y \in \dep(\Box p, p)\}\\
	& = \{x + 1 : x \in \dep(A, p)\}\\
	& = \dep(\Box A, p). 
\end{align*}
\end{proof}

\begin{lem}\label{lem4}
For any $i, k \in \omega$, if $\dep_n(A, p) = \{[i]_n\}$, then $\dep_n(A^k(p), p) = \{[ki]_n\}$. 
In particular, $\dep_n(A^n(p), p) = \{[0]_n\}$. 
\end{lem}
\begin{proof}
We prove by induction on $k$. 
For $k = 0$, since $A^0(p) \equiv p$ and $\dep(p, p) = \{0\}$, $\dep_n(A^0(p), p) = \{[0]_n\}$.  

Suppose $\dep_n(A^k(p), p) = \{[ki]_n\}$. 
Since $A^{k+1}(p) \equiv A(A^k(p))$, we obtain 
\[
	\dep_n(A^{k+1}(p), p) = \{[x]_n + [y]_n : x \in \dep(A, p)\ \text{and}\ y \in \dep(A^k, p)\}. 
\]
by Lemma \ref{lem2}. 
Since 	$\dep_n(A, p) = \{[i]_n\}$ and $\dep_n(A^k(p), p) = \{[ki]_n\}$, we have $\dep_n(A^{k+1}(p), p) = \{[i]_n + [ki]_n\} = \{[(k+1)i]_n\}$. 
\end{proof}

\section{Basic properties of $\wGL_n$}

In this section, we prove several basic properties of $\wGL_n$ used in our proof of the fixed point theorem of $\wGL_n$. 

\begin{prop}[See \cite{Sac01}]\label{trans}
For any modal formula $A$, 
\[
	\wGL_n \vdash \Box A \to \Box^{n+1} A.
\] 
\end{prop}
\begin{proof}
Since $A \to ((\Box^n A \land \Box^{2n} A) \to (A \land \Box^n A))$ is a tautology, 
\[
	\K \vdash A \to (\Box^n (A \land \Box^n A) \to (A \land \Box^n A)). 
\]
Then 
\[
	\K \vdash \Box A \to \Box (\Box^n (A \land \Box^n A) \to (A \land \Box^n A)). 
\]
By the axiom $\Box (\Box^n (A \land \Box^n A) \to (A \land \Box^n A)) \to \Box (A \land \Box^n A)$ of $\wGL_n$, we obtain 
\[
	\wGL_n \vdash \Box A \to \Box (A \land \Box^n A). 
\]
We conclude $\wGL_n \vdash \Box A \to \Box^{n+1} A$. 
\end{proof}

The following proposition is a variation of the result known as the substitution lemma that holds for $\GL$ (see Boolos \cite{Boo93}).

\begin{prop}\label{subst1}
For any modal formulas $A_j, B_j$ ($1 \leq j \leq m$) and $C(p_1, \ldots, p_m)$, 
\[
	\wGL_n \vdash \db{n}^+ \bigwedge_{j=1}^m (A_j \leftrightarrow B_j) \to (C(A_1, \ldots, A_m) \leftrightarrow C(B_1, \ldots, B_m)).
\]
\end{prop}
\begin{proof}
We prove by induction on the construction of $C$. 
We only prove the case that $C$ is of the form $\Box D$. 
By induction hypothesis, we have 
\[
	\wGL_n \vdash \db{n}^+ \bigwedge_{j=1}^m (A_j \leftrightarrow B_j) \to (D(A_1, \ldots, A_m) \leftrightarrow D(B_1, \ldots, B_m)).
\]
Then
\[
	\wGL_n \vdash \db{n+1} \bigwedge_{j=1}^m (A_j \leftrightarrow B_j) \to (\Box D(A_1, \ldots, A_m) \leftrightarrow \Box D(B_1, \ldots, B_m)).
\]
By Proposition \ref{trans}, 
\[
	\wGL_n \vdash \db{n} \bigwedge_{j=1}^m (A_j \leftrightarrow B_j) \to (\Box D(A_1, \ldots, A_m) \leftrightarrow \Box D(B_1, \ldots, B_m)).
\]
In particular, we conclude
\[
	\wGL_n \vdash \db{n}^+ \bigwedge_{j=1}^m (A_j \leftrightarrow B_j) \to (C(A_1, \ldots, A_m) \leftrightarrow C(B_1, \ldots, B_m)).
\]
\end{proof}

From our proof of Proposition \ref{subst1}, we also obtain the following proposition. 

\begin{prop}\label{subst2}
For any modal formulas $A_j, B_j$ ($1 \leq j \leq m$) and $C(p_1, \ldots, p_m)$, 
\[
	\wGL_n \vdash \db{n} \bigwedge_{j=1}^m (A_j \leftrightarrow B_j) \to (\Box C(A_1, \ldots, A_m) \leftrightarrow \Box C(B_1, \ldots, B_m)).
\]
\end{prop}

The following proposition says that a L\"ob-like rule holds in $\wGL_n$. 

\begin{prop}\label{fLob}
For any modal formula $A$, 
\[
\wGL_{n} \vdash \db{n}(\db{n}A \to A) \to \db{n}A.
\] 
\end{prop}

\begin{proof}
First, we prove by induction on $k$ that for any $k$ $(0 \leq k \leq n-1)$, 
\[
\wGL_{n} \vdash \db{n}(\db{n}A \to A) \to ((\Box^{k+1}A \land \cdots \land \Box^{n} A) \to \db{n}A). 
\]

If $k=0$, this is trivial. 

Assume that the statement holds for $k$, that is,
\[
\wGL_{n} \vdash \db{n}(\db{n}A \to A) \to ((\Box^{k+1}A \land \cdots \land \Box^{n} A) \to \db{n}A). 
\]
Then
\[
\wGL_{n} \vdash \db{n}(\db{n}A \to A) \to \Box((\Box^{k+1}A \land \cdots \land \Box^{n} A) \to \db{n}A). 
\]
By $\db{n}(\db{n}A \to A)$, we obtain
\[
\wGL_{n} \vdash \db{n}(\db{n}A \to A) \to \Box((\Box^{k+1}A \land \cdots \land \Box^{n} A) \to A). 
\]
Therefore 
\[
\wGL_{n} \vdash \db{n}(\db{n}A \to A) \to ((\Box^{k+2}A \land \cdots \land \Box^{n} A) \to \Box(\Box^{n} A \to A)). 
\]
By the axiom of $\wGL_{n}$, we have 
\begin{eqnarray}\label{eq1}
\wGL_{n} \vdash \db{n}(\db{n}A \to A) \to ((\Box^{k+2}A \land \cdots \land \Box^{n} A) \to \Box A). 
\end{eqnarray}
Whereas, by our assumption, 
\[
\wGL_{n} \vdash \db{n}(\db{n}A \to A) \to ((\Box^{k+2}A \land \cdots \land \Box^{n+1} A) \to (\Box^{2} A \land \cdots \land \Box^{n} A)). 
\]
Hence, we have
\[
\wGL_{n} \vdash \db{n}(\db{n}A \to A) \to ((\Box A \land \Box^{k+2}A \land \cdots \land \Box^{n} A) \to (\Box^{2} A \land \cdots \land \Box^{n} A)). 
\]
By combining this with (\ref{eq1}), we obtain
\[
\wGL_{n} \vdash \db{n}(\db{n}A \to A) \to ((\Box^{k+2}A \land \cdots \land \Box^{n} A) \to \db{n}A). 
\]
This means that the statement holds for $k + 1$.

For $k = n-1$, we have $\wGL_{n} \vdash \db{n}(\db{n}A \to A) \to (\Box^{n} A \to \db{n}A)$ and hence $\wGL_{n} \vdash \db{n}(\db{n}A \to A) \to \Box(\Box^{n} A \to \db{n}A)$. 
Since $\wGL_{n} \vdash \db{n}(\db{n}A \to A) \to \Box(\db{n}A \to A)$, we obtain $\wGL_{n} \vdash \db{n}(\db{n}A \to A) \to \Box(\Box^{n}A \to A)$. 
By the axiom of $\wGL_{n}$, we have $\wGL_{n} \vdash \db{n}(\db{n}A \to A) \to \Box A$. 
Therefore, we obtain $\wGL_{n} \vdash \db{n}(\db{n}A \to A) \to \db{n} A$. 
\end{proof}

From Proposition \ref{fLob}, we obtain the following corollary. 

\begin{cor}\label{Lob}
For any modal formula $A$, if $\wGL_{n} \vdash \db{n}A \to A$, then $\wGL_{n} \vdash A$. 
\end{cor}
\begin{prop}\label{equiv}
For any modal formulas $A$ and $B$, if $\wGL_n \vdash \Box^n A \to (A \leftrightarrow B)$, then $\wGL_n \vdash \Box A \leftrightarrow \Box B$. 
\end{prop}
\begin{proof}
Suppose $\wGL_n \vdash \Box^n A \to (A \leftrightarrow B)$. 
Then $\wGL_n \vdash A \to (\Box^n A \to B)$, and hence $\wGL_n \vdash \Box A \to (\Box^{n+1} A \to \Box B)$. 
Since $\wGL_n \vdash \Box A \to \Box^{n+1} A$ by Proposition \ref{trans}, $\wGL_n \vdash \Box A \to \Box B$. 

On the other hand, $\wGL_n \vdash B \to (\Box^n A \to A)$ by the supposition. 
Then  $\wGL_n \vdash \Box B \to \Box (\Box^n A \to A)$. 
Since $\wGL_n \vdash \Box (\Box^n A \to A) \to \Box A$, we obtain $\wGL_n \vdash \Box B \to \Box A$. 
\end{proof}

We can refine Proposition \ref{subst1} and Proposition \ref{subst2} by considering the sets $\dep_n(C, p)$. 

\begin{prop}\label{subst3}
Let $A, B$ and $C(p)$ be any modal formulas. 
\begin{enumerate}
	\item If $\dep_n(C, p) = \{[0]_n\}$, then 
\[
	\wGL_n \vdash (A \leftrightarrow B) \land \Box^n(A \leftrightarrow B) \to (C(A) \leftrightarrow C(B)).
\]
	\item If $\dep_n(C, p) = \{[i]_n\}$ for $0 < i < n$, then
\[
	\wGL_n \vdash \Box^i(A \leftrightarrow B) \to (C(A) \leftrightarrow C(B)).
\]
\end{enumerate}
\end{prop}
\begin{proof}
First, note that if $\dep(C, p) = \emptyset$, then 
\[
	\wGL_n \vdash (A \leftrightarrow B) \to (C(A) \leftrightarrow C(B))
\]
trivially holds. 

We prove clauses 1 and 2 simultaneously by induction on the construction of $C$. 

If $C$ is $p$, then $\dep(C, p) = \dep_n(C, p) = \{[0]_n\}$ and $\wGL_n \vdash (A \leftrightarrow B) \to (C(A) \leftrightarrow C(B))$.

Suppose that $C$ is of the form $D_0 \to D_1$. 
We only prove clause 1, and clause 2 is proved in a similar way. 
	If $\dep_n(C, p) = \{[0]_n\}$, then for each $j= 0, 1$, $\dep_n(D_j, p) = \{[0]_n\}$ or $\dep_n(D_j, p) = \emptyset$. 
	In either case, for each $j = 0, 1$, we have
	\[
		\wGL_n \vdash (A \leftrightarrow B) \land \Box^n (A \leftrightarrow B) \to (D_j(A) \leftrightarrow D_j(B))
	\]
	by induction hypothesis. 
	Then we obtain 
	\[
		\wGL_n \vdash (A \leftrightarrow B) \land \Box^n (A \leftrightarrow B) \to (C(A) \leftrightarrow C(B)). 
	\]

Suppose that $C$ is of the form $\Box D$ and $\dep_n(C, p) = \{[i]_n\}$ for $0 \leq i < n$. 
Let $j = i - 1$ if $i \neq 0$, and let $j = n -1$ if $i = 0$. 
Then $\dep_n(D, p) = \{[j]_n\}$. 
If $j \neq 0$, by induction hypothesis, we have
	\[
		\wGL_n \vdash \Box^j (A \leftrightarrow B) \to (D(A) \leftrightarrow D(B)). 
	\]
Then 
	\[
		\wGL_n \vdash \Box^{j+1} (A \leftrightarrow B) \to (C(A) \leftrightarrow C(B)). 
	\]
If $j = 0$, by induction hypothesis, 
	\[
		\wGL_n \vdash (A \leftrightarrow B) \land \Box^n (A \leftrightarrow B) \to (D(A) \leftrightarrow D(B)). 
	\]
Then
	\[
		\wGL_n \vdash \Box (A \leftrightarrow B) \land \Box^{n+1} (A \leftrightarrow B) \to (C(A) \leftrightarrow C(B)). 
	\]
Since $\wGL_n \vdash \Box (A \leftrightarrow B) \to \Box^{n+1} (A \leftrightarrow B)$ by Proposition \ref{trans},
	\[
		\wGL_n \vdash \Box (A \leftrightarrow B) \to (C(A) \leftrightarrow C(B)). 
	\]
In either case, we have obtained the required conclusion. 
\end{proof}

From our proof of Proposition \ref{subst3}, we also obtain the following proposition. 

\begin{prop}\label{subst4}
Let $A, B$ and $C(p)$ be any modal formulas. 
If $\dep_n(C, p) = \{[0]_n\}$ and $0 \notin \dep(C, p)$, then
\[
	\wGL_n \vdash \Box^n(A \leftrightarrow B) \to (C(A) \leftrightarrow C(B)).
\]
\end{prop}

Notice that for any modal formula $C$, $0 \notin \dep(C, p)$ if and only if $C$ is modalized in $p$.

\section{Effectively constructible fixed points in $\wGL_n$}

In this section, we prove the following main theorem of this paper. 

\begin{theorem}\label{FPT}
For any modal formula $A(p)$ which is modalized in $p$, a fixed point of $A(p)$ in $\wGL_n$ is effectively constructible.  
\end{theorem}

First, we show that for a proof of Theorem \ref{FPT}, we may consider only a certain restricted case, that is, it suffices to give an effective construction of fixed points of modal formulas which are of the form $\Box A(p)$. 
This reduction procedure is due to Linsdtr\"om \cite{Lin96}. 

\begin{theorem}\label{thm1}
For any modal formula $\Box A(p)$, a fixed point of $\Box A(p)$ in $\wGL_n$ is effectively constructible. 
\end{theorem}

\begin{lem}\label{reduction1}
Suppose that Theorem \ref{thm1} holds. 
Then for any modal formulas $\Box B_1(p_1, \ldots, p_m), \ldots, \Box B_m(p_1, \ldots, p_m)$, we can effectively find modal formulas $F_1, \ldots, F_m$ such that for any $i$ ($1 \leq i \leq m$), 
\[
	\wGL_n \vdash F_i \leftrightarrow \Box B_i(F_1, \ldots, F_m).
\]
\end{lem}
\begin{proof}
We prove by induction on $m$. 
The case of $m = 1$ is exactly Theorem \ref{thm1}. 
Suppose that the statement holds for $m$. 
Let $\Box B_1(p_1, \ldots, p_m, p_{m+1}), \ldots$, $\Box B_{m+1}(p_1, \ldots, p_m, p_{m+1})$ be any modal formulas. 
Then by induction hypothesis, we can effectively find modal formulas $F_1(p_{m+1}), \ldots, F_m(p_{m+1})$ such that for any $i$ ($1 \leq i \leq m$), 
\[
	\wGL_n \vdash F_i(p_{m+1}) \leftrightarrow \Box B_i(F_1(p_{m+1}), \ldots, F_m(p_{m+1}), p_{m+1}). 
\]
By Theorem \ref{thm1}, a fixed point $F$ of the formula $\Box B_{m+1}(F_1(q_{m+1}), \ldots, F_m(q_{m+1}), q_{m+1})$ with respect to $q_{m+1}$ can be found effectively. 
Then the modal formulas $F_1(F), \ldots, F_m(F)$ and $F$ satisfy the required condition. 
\end{proof}

\begin{lem}\label{reduction2}
Suppose that Theorem \ref{thm1} holds. 
Then Theorem \ref{FPT} holds. 
\end{lem}
\begin{proof}
Let $A(p)$ be any modal formula which is modalized in $p$. 
Then there exists a modal formula $\Box C_1(p), \ldots, \Box C_m(p)$ and $B(p_1, \ldots, p_m)$ such that $B(p_1, \ldots, p_m)$ does not contain $\Box$ and $A(p) \equiv B(\Box C_1(p), \ldots, \Box C_m(p))$. 
By Lemma \ref{reduction1}, we can effectively find modal formulas $F_1, \ldots, F_m$ such that for any $i$ ($1 \leq i \leq m$), 
\[
		\wGL_n \vdash F_i \leftrightarrow \Box C_i(B(F_1, \ldots, F_m)).
\]
Then $B(F_1, \ldots, F_m)$ is a fixed point of $A(p)$ in $\wGL_n$. 
\end{proof}

In the rest of this section, we prove Theorem \ref{thm1}. 
In the case of $\GL$, $\Box A(p)$ has a simple fixed point. 

\begin{fact}[See \cite{Sam76,Smo85}]\label{factGL}
For any modal formula $\Box A(p)$, 
\[
	\GL \vdash \Box A(\top) \leftrightarrow \Box A(\Box A(\top)).
\] 
\end{fact}

The fixed point theorem of $\GL$ immediately follows from Lemma \ref{reduction2} and Fact \ref{factGL}. 
In $\wGL_n$ for $n \geq 2$, a fixed point of $\Box A(p)$ is not so simple in general. 
However, we can prove the following proposition which is a counterpart of Fact \ref{factGL} in $\wGL_n$. 

\begin{prop}\label{fp1}
If $\dep_n(\Box A, p) = \{[0]_n\}$, then
\[
	\wGL_n \vdash \Box A(\top) \leftrightarrow \Box A(\Box A(\top)).
\]
\end{prop}
\begin{proof}
We have $\dep_n(A(\Box p), p) = \dep_n(\Box A, p) = \{[0]_n\}$ because by Lemma \ref{lem3}, $\dep(A(\Box p), p) = \dep(\Box A, p)$. 
It is obvious that $0 \notin \dep(A(\Box p), p)$. 
Then by Proposition \ref{subst4}, 
\begin{align*}
	\wGL_n \vdash \Box^n A(\top) & \to \Box^n (\top \leftrightarrow A(\top)) \\
	& \to (A(\Box \top) \leftrightarrow A(\Box A(\top))) \\
	& \to (A(\top) \leftrightarrow A(\Box A(\top))). 
\end{align*}
By Proposition \ref{equiv}, we conclude
\[
	\wGL_n \vdash \Box A(\top) \leftrightarrow \Box A(\Box A(\top)). 
\]
\end{proof}

Related to Proposition \ref{fp1}, in the next section, we show that if $\dep_n(\Box A, p)$ is a singleton, then $\Box A$ has the fixed point $(\Box A)^n(\top)$ in $\wGL_n$ (Theorem \ref{thm2}). 
In general, $\dep_n(\Box A, p)$ is more complex than a singleton, and so Proposition \ref{fp1} and Theorem \ref{thm2} do not guarantee the existence of fixed points of such a formula $\Box A$. 
To overcome this situation, we reduce the existence of fixed-points of such complex formulas to that of formulas with the depth $\{[0]_n\}$ modulo $n$.
Our strategy is to transform a formula of the form $\Box A$ into another formula so that the fixed points are preserved. 
In the following, we introduce two types of transformation of formulas. 
The first one is to remove $[0]_n$ from $\dep_n$ (Definition \ref{Def:0instance} and Lemma \ref{0instance}).
The second one is to shift the elements of $\dep_n$ (Definition \ref{Def:shifting} and Lemma \ref{shifting}). 

We introduce the first transformation of formulas. 


\begin{defn}\label{Def:0instance}
We say that a modal formula $\Box A'(p)$ is the {\it $0$-instance of $\Box A(p)$} if $\Box A'(p)$ is obtained by replacing all occurrences of $p$ in $\Box A$ whose depths are congruent to $0$ modulo $n$ with $\top$. 
\end{defn}

\begin{lem}\label{0instance}
For any modal formulas $\Box A(p)$ and $F$, if $F$ is a fixed point of the $0$-instance of $\Box A(p)$ in $\wGL_n$, then $F$ is also a fixed point of $\Box A(p)$ in $\wGL_n$. 
\end{lem}
\begin{proof}
If $[0]_n \notin \dep_n(\Box A(p), p)$, then the $0$-instance of $\Box A(p)$ is $\Box A(p)$ itself, and hence the lemma is trivial. 
We may assume $[0]_n \in \dep_n(\Box A(p), p)$. 
Let $q$ be some propositional variable not contained in $\Box A(p)$, and let $\Box B(p, q)$ be the modal formula obtained by replacing all occurrences of $p$ in $\Box A(p)$ whose depths are congruent to $0$ modulo $n$ with $q$. 
Then $\Box B(p, \top)$ is the $0$-instance of $\Box A(p)$. 
Since $\dep_n(\Box B, q) = \{[0]_n\}$, by Proposition \ref{fp1}, 
\[
	\wGL_n \vdash \Box B(p, \top) \leftrightarrow \Box B(p, \Box B(p, \top)). 
\]
Let $F$ be a fixed point of $\Box B(p, \top)$. 
Then 
\begin{align*}
	\wGL_n \vdash F & \leftrightarrow \Box B(F, \top) \\
	& \leftrightarrow \Box B(F, \Box B(F, \top)) \\
	& \leftrightarrow \Box B(F, F) \\
	& \equiv \Box A(F). 
\end{align*}
This means that $F$ is also a fixed point of $\Box A(p)$. 
\end{proof}

Before introducing the second transformation of formulas, we prove that an iterative substitution of $\Box A(p)$ into $p$ in $\Box A(p)$ preserves fixed points.

\begin{defn}\label{Def:ss}
Let $\Box A(p)$ be any modal formula. 
We say a sequence $\{\Box A_i\}_{i \in \omega}$ of modal formulas is a {\it $\Box A(p)$-substitution sequence} if the following conditions hold: 
\begin{enumerate}
	\item $\Box A_0 \equiv \Box A$. 
	\item $\Box A_{i+1}$ is obtained by replacing several occurrences of $p$ in $\Box A_i$ with $\Box A$. 
\end{enumerate}
\end{defn}

\begin{lem}\label{ss}
Let $\{\Box A_i\}_{i \in \omega}$ be a $\Box A(p)$-substitution sequence. 
Then for any $i \in \omega$ and any modal formula $F$, if $F$ is a fixed point of $\Box A_i$ in $\wGL_n$, then $F$ is also a fixed point of $\Box A$ in $\wGL_n$. 
\end{lem}
\begin{proof}
Since the lemma trivially holds for $i = 0$, we may assume $i > 0$. 
Suppose that $F$ is a fixed point of $\Box A_i$ in $\wGL_n$. 

Let $j$ be any natural number with $j < i$. 
By the definition of $\Box A(p)$-substitution sequences, there exists a modal formula $\Box B_j(p, q)$ which is obtained by replacing several occurrences of $p$ in $\Box A_j$ with $q$ such that $\Box A_{j+1}(p) \equiv \Box B_j(p, \Box A(p))$. 
By applying Proposition \ref{subst2} for $\Box B_j(F, q)$, we have 
\[
	\wGL_n \vdash \db{n}(F \leftrightarrow \Box A(F)) \to (\Box B_j(F, F) \leftrightarrow \Box B_j(F, \Box A(F))). 
\]
This means
\[
	\wGL_n \vdash \db{n}(F \leftrightarrow \Box A(F)) \to (\Box A_j(F) \leftrightarrow \Box A_{j+1}(F)). 
\]
Since this statement holds for all $j < i$, we have
\[
	\wGL_n \vdash \db{n}(F \leftrightarrow \Box A(F)) \to (\Box A(F) \leftrightarrow \Box A_i(F)). 
\]
Since $\wGL_n \vdash F \leftrightarrow \Box A_i(F)$, we obtain 
\[
	\wGL_n \vdash \db{n}(F \leftrightarrow \Box A(F)) \to (F \leftrightarrow \Box A(F)). 
\]
By Corollary \ref{Lob}, we conclude
\[
	\wGL_n \vdash F \leftrightarrow \Box A(F).
\]
\end{proof}

We introduce the second transformation of formulas. 

\begin{defn}\label{Def:shifting}
Let $k$ be any natural number. 
We say a $\Box A(p)$-substitution sequence $\{\Box A_i\}_{i \in \omega}$ is {\it $k$-shifting} if for each $i$, $\Box A_{i+1}$ is obtained by replacing all occurrences of $p$ in $\Box A_i$ whose depths are congruent to $k + i$ modulo $n$ with $\Box A$. 
\end{defn}

\begin{lem}\label{shifting}
Let $k$ be any natural number with $1 \leq k < n$. 
Suppose that 
\[
	\dep_n(\Box A, p) \subseteq \{[x]_n : 1 \leq x \leq k\}
\]
and let $\{\Box A_i\}_{i \in \omega}$ be the $k$-shifting $\Box A(p)$-substitution sequence. 
Then
\[
	\dep_n(\Box A_{n - k}, p) \subseteq \{[x]_n : 0 \leq x \leq k - 1\}.
\] 
\end{lem}
\begin{proof}
We prove by induction on $i$ that for all $i$ with $k + i \leq n$, 
\[
	\dep_n(\Box A_i, p) \subseteq \{[x]_n : 0 \leq x \leq k-1\ \text{or}\ k + i \leq x \leq n + k -1\}. 
\]
For $i = 0$, this is trivial because
\[
	\{[x]_n : 0 \leq x \leq k-1\ \text{or}\ k \leq x \leq n + k -1\} = \{[x]_n : 0 \leq x \leq n-1\}. 
\]
Suppose the statement holds for $i$. 
Assume $k + i + 1 \leq n$ and let $\Box B_i(p, q)$ be the modal formula obtained by replacing all occurrences of $p$ in $\Box A_i$ whose depths are congruent to $k + i$ modulo $n$ with $q$. 
Then $\Box A_{i+1} \equiv \Box B_i(p, \Box A)$, $\dep_n(\Box B_i, q) = \{[k + i]_n\}$ and
\[
	\dep_n(\Box B_i, p) \subseteq \{[x]_n : 0 \leq x \leq k-1\ \text{or}\ k + i + 1 \leq x \leq n + k -1\}. 
\]
By Lemma \ref{lem2}, $\dep(\Box B_i(p, \Box A), p)$ is equal to
\[
	 \{x + y : x \in \dep(\Box B_i, q)\ \text{and}\ y \in \dep(\Box A, p)\} \cup \dep(\Box B_i, p). 
\]
Hence $\dep_n(\Box A_{i+1}, p)$ is equal to
\[
	 \{[x]_n + [y]_n : x \in \dep(\Box B_i, q)\ \text{and}\ y \in \dep(\Box A, p)\} \cup \dep_n(\Box B_i, p). 
\]
Since $k + i + k \leq n + k - 1$, we have
\begin{align*}
	& \{[x]_n + [y]_n : x \in \dep(\Box B_i, q)\ \text{and}\ y \in \dep(\Box A, p)\}\\
	\subseteq & \{[k + i]_n + [y]_n : 1 \leq y \leq k\}\\
	= & \{[x]_n : k + i + 1 \leq x \leq k + i + k\}\\
	\subseteq & \{[x]_n : k + i + 1 \leq x \leq n + k - 1\}. 
\end{align*}
Hence we obtain
\[
	\dep_n(\Box A_{i+1}, p) \subseteq \{[x]_n : 0 \leq x \leq k-1\ \text{or}\ k + i + 1 \leq x \leq n + k -1\}. 
\]

In particular, for $i = n - k$, 
\[
	\dep_n(\Box A_{n - k}, p) \subseteq \{[x]_n : 0 \leq x \leq k-1\ \text{or}\ n \leq x \leq n + k -1\}.
\]
Since $\{[x]_n : n \leq x \leq n + k -1\} = \{[x]_n : 0 \leq x \leq k -1\}$, we conclude
\[
	\dep_n(\Box A_{n - k}, p) \subseteq \{[x]_n : 0 \leq x \leq k-1\}.
\] 
\end{proof}

We finish our proof of Theorem \ref{thm1}. 

\begin{proof}[Proof of Theorem \ref{thm1}]
We define a sequence $\Box B_0(p), \ldots, \Box B_{n-1}(p)$ of modal formulas recursively as follows: 
\begin{enumerate}
	\item $\Box B_0(p) :\equiv \Box A(p)$. 
	\item Suppose that $\Box B_k(p)$ is already defined. 
	Let $\Box B_k'(p)$ be the $0$-instance of $\Box B_k(p)$, and let $\{\Box C_{k, i}(p)\}_{i \in \omega}$ be the $(n - k - 1)$-shifting $\Box B_k'(p)$-substitution sequence. 
Define $\Box B_{k+1}(p) : \equiv \Box C_{k, k+1}(p)$. 
\end{enumerate}

We prove that $\Box B_{n-1}(\top)$ is a fixed point of $\Box A(p)$ in $\wGL_n$. 
For this, we prove by induction on $k$ that for all $k < n$, the following two conditions hold: 
\begin{enumerate}
	\item $\dep_n(\Box B_k(p), p) \subseteq \{[x]_n : 0 \leq x \leq n - k - 1\}$. 
	\item Every fixed point of $\Box B_k(p)$ in $\wGL_n$ is also a fixed point of $\Box A(p)$ in $\wGL_n$. 
\end{enumerate}
For $k = 0$, these are trivial. 
We suppose that the two conditions hold for $k$. 
Assume $k + 1 < n$. 

1. By the definition of the $0$-instances, $\dep_n(\Box B_k'(p), p) \subseteq \{[x]_n : 1 \leq x \leq n - k - 1\}$. 
Since $\{\Box C_{k, i}(p)\}_{i \in \omega}$ is $(n - k - 1)$-shifting, 
\[
	\dep_n(\Box C_{k, n - (n - k - 1)}(p), p) \subseteq \{[x]_n : 0 \leq x \leq n - k - 2\}
\]
by Lemma \ref{shifting}. 
Since $\Box B_{k+1}(p) \equiv \Box C_{k, k+1}(p)$, we obtain
\[
	\dep_n(\Box B_{k+1}(p), p) \subseteq \{[x]_n : 0 \leq x \leq n - (k+1) - 1\}. 
\]
	
2. Let $F$ be any fixed point of $\Box B_{k+1}(p)$ in $\wGL_n$. 
Since $\{\Box C_{k, i}(p)\}_{i \in \omega}$ is a $\Box B_k'(p)$-substitution sequence, $F$ is also a fixed point of $\Box B_k'(p)$ by Lemma \ref{ss}. 
Then $F$ is a fixed point of $\Box B_k(p)$ by Lemma \ref{0instance}. 
Therefore $F$ is also a fixed point of $\Box A(p)$ by induction hypothesis. 

In particular, $\dep_n(\Box B_{n-1}(p), p) \subseteq \{[0]_n\}$. 
Then $\Box B_{n-1}(\top)$ is a fixed point of $\Box B_{n-1}(p)$ in $\wGL_n$ by Proposition \ref{fp1}. 
Therefore $\Box B_{n-1}(\top)$ is also a fixed point of $\Box A(p)$ in $\wGL_n$. 
\end{proof}

\section{Examples}

In our proof of Theorem \ref{thm1}, we gave an effective procedure for constructing fixed points in $\wGL_n$. 
More precisely, from an input $\Box A(p)$, we constructed the sequence $\Box B_0(p), \Box B_0'(p), \Box B_1(p), \Box B_1'(p), \ldots, \Box B_{n-1}(p)$ of modal formulas, and then we concluded that the modal formula $\Box B_{n-1}(\top)$ is a fixed point of $\Box A(p)$ in $\wGL_n$. 

For example, we execute this procedure for the cases $n = 2$ and $n = 3$. 

\begin{ex}[$\wGL_2$]
Let $\Box A(p)$ be any modal formula and let $\Box B(p, q)$ be the modal formula obtained by replacing all occurrences of $p$ in $\Box A(p)$ whose depths are congruent to $0$ modulo $2$ with $q$. 
\begin{enumerate}
	\item $\Box B_0(p) \equiv \Box A(p)$.
	\item $\Box B_0'(p) \equiv \Box B(p, \top)$.
	\item $\Box B_1(p) \equiv \Box B(\Box B(p, \top), \top)$. 
	\item $\Box B_1(\top) \equiv \Box B(\Box A(\top), \top)$ is a fixed point of $\Box A(p)$ in $\wGL_2$. 
\end{enumerate}
\end{ex}

The case of $\wGL_3$ is slightly complicated. 

\begin{ex}[$\wGL_3$]\label{Ex:compl}
Let $\Box A(p)$ be any modal formula and let $\Box B(p_2, p_1, p_0)$ be the modal formula obtained by replacing all occurrences of $p$ in $\Box A(p)$ whose depths are congruent to $i$ modulo $3$ with $p_i$ for every $0 \leq i \leq 2$.
\begin{enumerate}
	\item $\Box B_0(p) \equiv \Box A(p)$. 
	\item $\Box B_0'(p) \equiv \Box B(p, p, \top)$. 
	\item $\Box B_1(p) \equiv \Box B(\Box B(p, p, \top), p, \top)$.
	\item $\Box B_1'(p) \equiv \Box B(\Box B(p, \top, \top), p, \top)$.  
	\item $\Box B_2(p) \equiv \Box B_1'(\Box B_1'(\Box B_1'(p)))$. 
	\item $\Box B_2(\top) \equiv \Box B_1'(\Box B_1'(\Box B_1'(\top)))$ is a fixed point of $\Box A(p)$ in $\wGL_3$. 
\end{enumerate}
\end{ex}

As shown in this example, in general, fixed points of $\Box A(p)$ in $\wGL_n$ are complicated. 
On the other hand, Proposition \ref{fp1} states that if $\dep_n(\Box A, p) = \{[0]_n\}$, then $\Box A(p)$ has a simple fixed point $\Box A(\top)$ in $\wGL_n$. 
Moreover, we prove that if $\dep_n(\Box A, p)$ is a singleton, then $\Box A(p)$ also has a simple fixed point. 

\begin{theorem}\label{thm2}
If $\dep_n(\Box A, p) = \{[i]_n\}$ for $0 < i < n$, then
\[
	\wGL_n \vdash (\Box A)^n(\top) \leftrightarrow (\Box A)^{n+1}(\top).
\]
\end{theorem}
\begin{proof}
By Lemma \ref{lem4}, $\dep_n((\Box A)^n(p), p) = \{[0]_n\}$. 
Let $C(p)$ be the modal formula $ A((\Box A)^{n-1})(\Box p)$. 
Since $(\Box A)^n(p) \equiv \Box A((\Box A)^{n-1}(p)) \equiv \Box A((\Box A)^{n-1})(p)$, we obtain $\dep_n(C(p), p) = \{[0]_n\}$ by Lemma \ref{lem3}. 
Also $0 \notin \dep(C(p), p)$. 

\vspace{0.1in}
{\bf Claim.} $\wGL_n \vdash (\Box A)^n(\top) \leftrightarrow (\Box A)^{2n}(\top)$.

\begin{proof}[Proof of Claim]
Since 
\[
	\wGL_n \vdash \Box^n C(\top) \to \Box^n (\top \leftrightarrow C(\top)), 
\]
we obtain
\begin{eqnarray}\label{eq2}
	\wGL_n \vdash \Box^n C(\top) \to (C(\top) \leftrightarrow C(C(\top)))
\end{eqnarray}
by Proposition \ref{subst4}. 
Then by Proposition \ref{equiv}, we obtain 
\[
	\wGL_n \vdash \Box C(\top) \leftrightarrow \Box C(C(\top)). 
\]

Here 
\begin{align*}
	\wGL_n \vdash \Box C(\top) & \leftrightarrow \Box A((\Box A)^{n-1})(\Box \top) \\
	& \leftrightarrow (\Box A)^n(\Box \top) \\
	& \leftrightarrow (\Box A)^n(\top).
\end{align*}
Also 
\begin{align*}
	\wGL_n \vdash \Box C(C(\top)) & \leftrightarrow \Box A((\Box A)^{n-1})(\Box C(\top)) \\
	& \leftrightarrow (\Box A)^n(\Box C(\top))\\
	& \leftrightarrow (\Box A)^n((\Box A)^n(\top))\\
	& \leftrightarrow (\Box A)^{2n}(\top). 
\end{align*}
We conclude
\[
	\wGL_n \vdash(\Box A)^n(\top) \leftrightarrow (\Box A)^{2n}(\top).
\]
\end{proof}

Since
\[
	\wGL_n \vdash \Box^n A((\Box A)^n)(\top) \to \Box^n(\top \leftrightarrow A((\Box A)^n)(\top)), 
\]
we also have
\[
	\wGL_n \vdash \Box^n A((\Box A)^n)(\top) \to (C(\top) \leftrightarrow C(A((\Box A)^n)(\top)))
\]
by Proposition \ref{subst4}. 
Here 
\begin{align*}
	\wGL_n \vdash C(A((\Box A)^n)(\top)) & \leftrightarrow A((\Box A)^{n-1})(\Box A((\Box A)^n))(\top) \\
	& \leftrightarrow A((\Box A)^{2n})(\top) \\
	& \leftrightarrow A((\Box A)^n)(\top)
\end{align*}
by Claim. 
Hence
\[
	\wGL_n \vdash \Box^n A((\Box A)^n)(\top) \to (C(\top) \leftrightarrow A((\Box A)^n)(\top)). 
\]
By Proposition \ref{equiv}, 
\[
	\wGL_n \vdash \Box C(\top) \leftrightarrow \Box A((\Box A)^n)(\top). 
\]
This means
\[
	\wGL_n \vdash (\Box A)^n(\top) \leftrightarrow (\Box A)^{n+1}(\top). 
\]
\end{proof}

We close this paper with the following example showing that our fixed points given in this paper might not be simplest.  

\begin{ex}\label{Ex:efp}
Let $A(p)$ be $\Box \neg p$. 
Since $\dep_3(\Box A, p) = \{[2]_3\}$, $\Box A(\Box A(\Box A(\top)))$ is a fixed point of $\Box A(p)$ in $\wGL_3$ by Theorem \ref{thm2}. 
Moreover, we show that $\Box A(\Box A(\top))$ is also a fixed point. 
Here 
\begin{itemize}
	\item $\Box A(\Box A(\top)) \equiv \Box^2 \Diamond^2 \top$, 
	\item $\Box A(\Box A(\Box A(\top))) \equiv \Box^2 \Diamond^2 \Box^2 \bot$. 
\end{itemize}
Then it suffices to prove $\wGL_3 \vdash \Box^2 \Diamond^2 \top \leftrightarrow \Box^2 \Diamond^2 \Box^2 \bot$. 

$(\leftarrow)$: 
Since $\K \vdash \Box^2 \bot \to \top$, it is easy to derive $\K \vdash \Diamond^2 \Box^2 \bot \to \Diamond^2 \top$ and $\K \vdash \Box^2 \Diamond^2 \Box^2 \bot \to \Box^2 \Diamond^2 \top$. 

$(\to)$: Notice that for any modal formulas $B$ and $C$, $\K \vdash \Diamond^k B \land \Box^k C \to \Diamond^k(B \land C)$. 
Let $D$ be the modal formula $\Diamond^2 \top \land \Box^2 \Diamond^2 \top \land \Box^3 \Diamond^2 \top$. 
Then by Proposition \ref{trans}, 
\begin{align*}
	\wGL_3 \vdash D & \to \Diamond^2 \top \land \Box^2 \Diamond^2 \top \land \Box^3 \Diamond^2 \top \land \Box^5 \Diamond^2 \top \land \Box^6 \Diamond^2 \top\\
	& \to \Diamond^2 (\Diamond^2 \top \land \Box \Diamond^2 \top \land \Box^3 \Diamond^2 \top \land \Box^4 \Diamond^2 \top) \\
	& \to \Diamond^3 D.
\end{align*}
We have proved $\wGL_3 \vdash \Box^3 \neg D \to \neg D$. 
By Corollary \ref{Lob}, $\wGL_3 \vdash \neg D$. 
Also 
\begin{align*}
	\wGL_3 \vdash \Box^2 \Diamond^2 \top \land \Diamond^2 \Box^2 \Diamond^2 \top & \to \Box^2 \Diamond^2 \top \land \Diamond^2 \Box^2 \Diamond^2 \top \land \Box^5 \Diamond^2 \top\\
	& \to \Diamond^2 D. 
\end{align*}
Since $\wGL_3 \vdash \neg \Diamond^2 D$, we have $\wGL_3 \vdash \Box^2 \Diamond^2 \top \to \neg  \Diamond^2 \Box^2 \Diamond^2 \top$. 
This means $\wGL_3 \vdash \Box^2 \Diamond^2 \top \to \Box^2 \Diamond^2 \Box^2 \bot$. 
\end{ex}

\section{Future works}

As shown in Example \ref{Ex:compl}, our method produces slightly complex fixed points. On the other hand, Example \ref{Ex:efp} indicates that the existence of a method  producing simpler fixed points. 
We propose the following problem.

\begin{prob}
Does there exist a method to produce simpler fixed points of formulas in $\wGL_n$ than that obtained by our method?
\end{prob}

It is easily shown by a semantical argument that the intersection of all $\wGL_n$'s is exactly the logic $\K$. 
Thus there exists no minimum normal logic among the normal modal logics for which the fixed point theorem holds because the fixed point theorem does not hold for $\K$.
Relating to this observation, we propose the following problem.

\begin{prob}\label{prob2}
Is there a minimal modal logic among the normal modal logics for which the fixed point theorem holds?
\end{prob}

A negative answer to Problem \ref{prob2} follows from an affirmative answer to the following problem.
\begin{prob}
Does the fixed point theorem hold for the intersection of two logics for which the fixed point theorem holds?
\end{prob}

\bibliographystyle{plain}
\bibliography{ref}

\begin{thebibliography}{10}

\bibitem{Boo93}
George Boolos.
\newblock {\em The logic of provability}.
\newblock Cambridge University Press, Cambridge, 1993.

\bibitem{Kur18}
Taishi Kurahashi.
\newblock Arithmetical soundness and completeness for {$\Sigma_2$} numerations.
\newblock {\em Studia Logica}, 106(6):1181--1196, 2018.

\bibitem{Lin96}
Per Lindstr{\"o}m.
\newblock Provability logic---a short introduction.
\newblock {\em Theoria}, 62(1-2):19--61, 1996.

\bibitem{Lin06}
Per Lindstr{\"o}m.
\newblock Note on some fixed point constructions in provability logic.
\newblock {\em The Journal of Philosophical Logic}, 35(3):225--230, 2006.

\bibitem{Rei90}
Lisa Reidhaar-Oslon.
\newblock A new proof of the fixed-point theorem of provability logic.
\newblock {\em Notre Dame Journal of Formal Logic}, 31(1):37--48, 1990.

\bibitem{Sac01}
Lorenzo Sacchetti.
\newblock The fixed point property in modal logic.
\newblock {\em Notre Dame Journal of Formal Logic}, 42(2):65--86, 2001.

\bibitem{Sam76}
Giovanni Sambin.
\newblock An effective fixed-point theorem in intuitionistic diagonalizable
  algebras.
\newblock {\em Studia Logica}, 35(4):345--361, 1976.

\bibitem{SV82}
Giovanni Sambin and Silvio Valentini.
\newblock The modal logic of provability. {T}he sequential approach.
\newblock {\em Journal of Philosophical Logic}, 11(3):311--342, 1982.

\bibitem{Smo78}
Craig Smory{\'n}ski.
\newblock Beth's theorem and self-referential sentences.
\newblock In {\em Logic Colloquium '77}, pages 253--261. North-Holland,
  Amsterdam-New York, 1978.

\bibitem{Smo85}
Craig Smory{\'n}ski.
\newblock {\em Self-reference and modal logic}.
\newblock Universitext. Springer-Verlag, New York, 1985.

\bibitem{Sol76}
Robert~M. Solovay.
\newblock Provability interpretations of modal logic.
\newblock {\em Israel Journal of Mathematics}, 25(3-4):287--304, 1976.

\end{thebibliography}

\end{document}